\documentclass[11pt, a4paper]{amsart}
\usepackage{a4wide}
\usepackage{amssymb}
\usepackage{amsthm}
\usepackage{amsfonts}
\newtheorem{lemma}{\bf Lemma}[section]

\newtheorem{proposition}[lemma]{\bf Proposition}
\newtheorem{fact}[lemma]{\bf Fact}

\newcommand{\Pz}{{\mathcal P}_2}
\newcommand{\card}{{\mathrm{card}}}

\begin{document}

\title{Hadwiger's conjecture for graphs with infinite chromatic number}

\author{Dominic van der Zypen}
\address{M\&S Software Engineering, Morgenstrasse 129, CH-3018 Bern,
Switzerland}
\email{dominic.zypen@gmail.com}


\begin{abstract} We construct a connected graph
$H$ such that \begin{enumerate}
\item $\chi(H) = \omega$; \item $K_\omega$, the complete graph on $\omega$
points, is not a minor of $H$.\end{enumerate} 
Therefore Hadwiger's conjecture
does not hold for graphs with infinite coloring number.
\end{abstract}

\maketitle
\parindent = 0mm
\parskip = 2 mm
\section{Notation}
In this note we are only concerned with simple
undirected graphs $G = (V, E)$ 
where $V$ is a set and $E \subseteq \Pz(V)$ where
$$\Pz(V) = \big\{ \{x,y\} : x,y \in V\textrm{ and } x\neq y\big\}.$$ We
also require that $V\cap E = \emptyset$ to avoid notational ambiguities.
We denote the vertex set of a graph $G$ by $V(G)$ and the edge set by 
$E(G)$. Moreover, for any cardinal $\alpha$ we denote the complete
graph on $\alpha$ points by $K_\alpha$. 

For any graph $G$, disjoint subsets $S, T \subseteq V(G)$ 
are said to be {\em connected
to each other} if there are $s \in S, t\in T$ with $\{s,t\}\in E(G)$.
Note that $K_\alpha$ is a {\em minor} of a graph $G$ if and only if 
there is 
a collection $\{S_\beta: \beta \in \alpha\}$ of nonempty, connected and 
pairwise disjoint subsets of $V(G)$ such that for all $\beta,\gamma \in \alpha$
with $\beta \neq \gamma$ the sets $S_\beta$ and $S_\gamma$ are connected
to each other. We will need the following observation later on:
\begin{fact}\label{connectedness}For any graph $G$, finite or infinite, the 
following are equivalent:\begin{enumerate}
\item $G$ is connected;
\item if $S,T \subseteq V(G)$ are nonempty and disjoint
such that $S\cup T = V(G)$ then $S, T$ are connected to each other.
\end{enumerate}
\end{fact}
\section{The construction}
In \cite{Ha}, Hadwiger formulated his well-known and deep conjecture, 
linking the chromatic number $\chi(G)$ of a graph $G$ with 
clique minors. His conjecture can be formulated that $K_{\chi(G)}$ is 
a minor of $G$ for every graph $G$. In the following we present a connected 
graph $H$ with chromatic number $\omega$ such that $K_\omega$ is not a
minor of $H$. Let $\mathbb{N}$ be the set of positive integers. For any 
$n\in \mathbb{N}$ we let
$$C_n = \{1,\ldots, n\}\times\{n\}$$ and set 
$V(H) = \bigcup_{n\in\mathbb{N}} C_n.$ As for the edge set of $H$, we define
$$E(H) = \big\{\{(1,n), (1,n+1)\}: n\in \mathbb{N}\big\} \cup
\bigcup_{n\in\mathbb{N}}\Pz(C_n).$$
\begin{proposition}$\chi(H) = \omega.$\end{proposition}
\begin{proof}Since we have $\card(V(H)) = \omega$ we get $\chi(H) 
\leq \omega$. Moreover, each $C_n$ is a complete subgraph of $H$, so
$H$ cannot be colored with finitely many colors. \end{proof}

For the remainder of this note, we assume that $\{S_n: n \in \omega\}$
is a collection of nonempty, connected, pairwise disjoint subsets
of $H$ such that for $m\neq n$ the sets $S_n, S_m$ are connected
to each other. Our goal is to show that such a collection cannot exist.

First, we need a simple observation
on what a connected subset of $H$ looks like. If $S\subseteq V(H)$ we
define $I(S) = \{n\in \mathbb{N}: C_n\cap S \neq \emptyset\}$.
\begin{lemma}\label{interval}
Suppose $S\subseteq V(H)$ is connected and $m<n \in I(S)$.
Then for all $x\in \mathbb{N}$ with $m\leq x\leq n$ we have $(1,x) \in S$.
\end{lemma}
\begin{proof} If $(1,m)\notin S$ then $T=S\cap C_m$ and $S\setminus T$
are disjoint, nonempty and not connected to each other. By Fact 
\ref{connectedness}, $S$ is not connected, contradicting our assumption.
A similar argument shows that $(1,n)\in S$. 
Suppose there is $x$ with $m<x<n$ and
$(1,x) \notin S$. Then set $T = \{(i,j) \in S: j < x\}$.
Again, $T$ and $S\setminus T$ are nonempty and not connected to each 
other, so $S$ is not connected, contradicting our assumption.
\end{proof}
If $\{S_n: n \in \omega\}$ is a collection of subsets of $V(H)$ as
described above, then for every $k\in \mathbb{N}$ the set of 
neighbors of $S_k$, which is denoted by $N(S_k)$, must be infinite.
As the next lemma shows, this implies that $I(S_k)$ must be infinite 
for all $k\in \mathbb{N}$.
\begin{lemma}\label{infinite}
If $S\subseteq V(H)$ is such that $I(S)$ is finite,
then $N(S)$ is finite. 
\end{lemma}
\begin{proof}
Let $m= \max(I_S)$. Then 
$N(S) \subseteq \bigcup_{i=1}^{m+1}C_i$, which is a finite set. 
\end{proof}
Now we go back to our assumption that 
$\{S_n: n \in \omega\}$ 
is a collection of nonempty, connected, pairwise disjoint subsets
of $H$ such that for $m\neq n$ the sets $S_n, S_m$ are connected
to each other. 
We consider just two of these sets, say $S_0, S_1$. Because of lemma
\ref{infinite}, the sets $I(S_0)$ and $I(S_1)$ are infinite. For
$k = 0,1$ let $\mu_k = \min(I(S_i))$. We may assume that $\mu_0 \leq \mu_1$.
Since $I(S_0)$ is infinite, there is $n\in I(S_0)$ with $n \geq \mu_1$.
So lemma \ref{interval} implies that $(1,\mu_1) \in S_0 \cap S_1$,
contradicting the assumption that the $S_k$ are pairwise disjoint.
So we established:
\begin{proposition} The complete graph $K_\omega$ is not a minor of $H$.
\end{proposition}
\end{document}